\documentclass{amsart}

\title{The Joint Embedding Property and Maximal Models}
\author[J. Baldwin]{John T. Baldwin}
\address{Department of Mathematics, Statistics, and Computer Science, M/C 249  851 S. Morgan Chicago IL 60607}
\email{jbaldwin@uic.edu}

\author[M. Koerwien]{Martin Koerwien}
\email{kwienmart@gmail.com}

\author[I. Souldatos]{Ioannis  Souldatos}
\address{4001 W.McNichols, Mathematics Department, University of Detroit Mercy, Detroit, MI 48221, USA}
\email{souldaio@udmercy.edu}

\date{\today}

\usepackage[latin1]
{inputenc}
\usepackage{geometry,amssymb,amsthm,amsmath,
graphics,enumerate}
\usepackage{enumerate}
\usepackage{url}
\usepackage{hyperref}

\theoremstyle{plain}
\newtheorem{Thm}{Theorem}[subsection]
\newtheorem{Lem}[Thm]{Lemma}
\newtheorem{question}[Thm]{Question}
\newtheorem{Cor}[Thm]{Corollary}

\newtheorem{Def}[Thm]{Definition}

\newtheorem{Ex}[Thm]{Example}
\newtheorem{Not}[Thm]{Notation}
\newtheorem{Fact}[Thm]{Fact}
\newtheorem{obs}[Thm]{Observation}
 \newtheorem{assumption}[Thm]{Assumption}

\newcommand{\M}{\ensuremath{\mathcal{M}}}
\newcommand{\N}{\ensuremath{\mathcal{N}}}

\newcommand{\Mscr}{\ensuremath{\mathcal{M}}}

\newcommand{\omegaone}{\ensuremath{\omega_1}}
\newcommand{\lomegaone}{\ensuremath{L_{\omega_1,\omega}}}

\def\hbar{{\bf h}}

\def\bK{\mbox{\boldmath $K$}}
\def\bkzero{\mbox{\boldmath $K$}_0}
\def\subm{\prec_{\mbox{\scriptsize \boldmath $K$}}}
\def\subhat{\prec_{\hat \bK_0}}
\def\subKzero{\prec_{\bkzero}}

\begin{document}

\begin{abstract} We
introduce the notion of a `pure' Abstract Elementary Class to block trivial
counterexamples. We study classes of models of bipartite graphs and
show:

\textbf{Main Theorem} (cf. Theorem \ref{manymax} and Corollary \ref{manymaxap}):  If $\langle \lambda_i: i\le \alpha <
\aleph_1\rangle$ is a strictly increasing
 sequence of characterizable cardinals (Definition~\ref{Def:CharacterizableCard}) whose models satisfy JEP$(<\lambda_0)$,
there is an $\lomegaone$-sentence $\psi$ whose models
 form a pure AEC and
\begin{enumerate}
\item The models of $\psi$ satisfy JEP$(<\lambda_0)$, while JEP
    fails for all larger cardinals and AP fails in all infinite
cardinals.
\item There exist $2^{\lambda_i^+}$ non-isomorphic maximal models
    of $\psi$ in $\lambda_i^+$, for all $i\le\alpha$, but no
    maximal models in any other cardinality; and
\item $\psi$ has arbitrarily large models.
\end{enumerate}

In particular this shows the Hanf number for JEP and the
Hanf number for maximality for pure AEC with L\"{o}wenheim number
$\aleph_0$ are at least $\beth_{\omega_1}$. We show that although
$AP(\kappa)$ for each $\kappa$ implies the full amalgamation
property, $JEP(\kappa)$ for each $\kappa$ does not imply the full joint embedding property.

We prove the main combinatorial device of this paper cannot be
used to extend the main theorem to a complete sentence.
\end{abstract}



\maketitle

We investigate in this paper the spectra of joint embedding and of
maximal models for an Abstract Elementary Class (AEC), in particular
for AEC defined by universal $\lomegaone$-sentences under
substructure. Our main result provides a collection of bipartite
graphs whose combinatorics allows us to construct for any given
countable strictly increasing sequence of characterizable cardinals
$(\lambda_i)$, a sentence of $\lomegaone$ whose models have joint
embedding below $\lambda_0$ and $2^{{\lambda_i}^+}$-many maximal
models in each $\lambda_i^+$, but arbitrarily large models. Two
examples of such sequences $(\lambda_i)$ are: (1) an enumeration of
an arbitrary countable subset of the $\beth_{\alpha}$,
$\alpha<\omegaone$, and (2) an enumeration of an arbitrary countable
subset of the $\aleph_n$, $n<\omega$.

We give precise definitions and more details in Section~\ref{pure}.
In Section~\ref{section:intro}, we describe our basic combinatorics
and the main constructions are in Section~\ref{incom}. We now provide
some background explaining several motivations for this study.

In first order logic, work from the 1950's deduces syntactic
characterizations of such properties as joint embedding and
amalgamation via the compactness theorem.  The syntactic conditions
immediately yield that if these properties hold in one cardinality
they hold in all cardinalities. For AEC this situation is vastly
different.  In fact, one major stream studies what are sometimes
called J\'{o}nsson classes that satisfy: amalgamation, joint
embedding, and have arbitrarily large models. (See, for example,
\cite{Sh394,GrVaDupcat,Baldwincatmon} and a series of paper such as
\cite{GrVaDupcat}.) Without this hypothesis the properties must be
parameterized and the relationship between, e.g.\! the Joint
Embedding Property (JEP) holding in various cardinals, becomes a
topic for study.  In \cite{GrossbergBilgi} Grossberg conjectures the
existence of a Hanf number for the Amalgamation Property (AP): a
cardinal $\mu(\lambda)$ such that if an AEC with L\"{o}wenheim number
$\lambda$ has the AP in some cardinal greater than $\mu(\lambda)$
then it has the amalgamation property in all larger cardinals. Boney
\cite{Boneyct} makes great progress on this problem by showing that
if $\kappa$ is strongly compact and an AEC $\bK$ is categorical in
$\lambda^+$ for some $\lambda \geq \kappa$, then $\bK$ has JEP and AP
above $\kappa$. Baldwin and Boney, \cite{BBHanf}, show that if there
is a strongly compact cardinal then it is an upper bound on the Hanf
number for joint embedding.  Our results here give much smaller but
concrete lower bounds in ZFC for the Hanf number of JEP, again with
no categoricity involved.

The lack of a syntactic condition for joint embedding or amalgamation
is a symptom of the lack of a good notion of `complete' for AEC's. In
particular  various `trivial' counterexamples arise from mere  (or
slightly disguised) disjunction of sentences. We introduce the notion
of a pure AEC to address this issue. We attempted in \cite{BHK1,
BHK2} to find other notions of `complete' which might provide a more
robust substitute for  `joint embedding'. Often  `complete for
$L_{\omega_1,\omega}$' is taken as a good completeness notion.  It
certainly is a robust notion as witnessed by Shelah's categoricity
theorem for such sentences. But the examples of e.g \cite{BFKL,BLS}
show that even such sentences do not guarantee even joint
embedding in all cardinals. Here we are considering weakenings of the
full joint embedding property. But we show (Theorem~\ref{nocomplete}),
that (even weak versions of) Lemma~\ref{manymax} cannot be extended
to a complete sentence for the combinatorics here.

Kolesnikov and Lambie-Hanson \cite{KLH} study a  family of AEC's
called coloring classes,  They show  that for these
classes\footnote{For easy comparison, we restrict their results to
countable languages.} the amalgamation property is equivalent to
disjoint amalgamation and the Hanf number for amalgamation is
$\beth_{\omega_1}$; specific classes fail disjoint amalgamation for the first time
arbitrarily close to $\beth_{\omega_1}$.  Our examples have
arbitrarily large models and no maximal model above
$\beth_{\omega_1}$.  But specific classes  have maximal models
arbitrarily close to $\beth_{\omega_1}$; we specify the cardinalities
of the maximal models exactly.

Baldwin, Koerwien, and Laskowski \cite{BKL} exploit excellence in appropriate classes, axiomatized by universal sentences in
$L_{\omega_1,\omega}$ to construct complete sentences in $\lomegaone$
that uniformly homogeneously characterize ($\phi_\alpha$ has no model
of cardinality $>\aleph_\alpha$) cardinals below $\aleph_{\omega}$.
We use these examples as an input to Corollary~\ref{manymaxap}(1)
constructing AEC that have maximal models in a countable set of
cardinals less than $\aleph_\omega$. Similarly we use Morley's
example to show the Hanf number for JEP is at least
$\beth_{\omega_1}$ in Corollary~\ref{manymaxap}(2).

\section{JEP and Pure AEC}\label{pure}
\numberwithin{Thm}{section}

One can trivially augment any AEC $\bK$ by adding structures below
the L\"{o}wenheim-Skolem number LS($\bK$) which have no extensions; to avoid such
trivialities the following assumption applies to the rest of the
paper.

\begin{assumption}\label{BasicAssumption}
For each AEC $(\bK,\subm)$, we work at or above the L\"{o}wenheim-Skolem number LS($\bK$).
\end{assumption}

In this section we spell out the parameterized notions of joint
embedding and introduce the notion of pure and hybrid AEC.  We then
show that there is no real theory possible if hybrid AEC are allowed.

\begin{Def}\label{jepdef} The AEC $(\bK,\subm)$ has  the {\em joint embedding
property}
at the infinite cardinal $\kappa$ (JEP$(\kappa)$)  if for any two
models $A,B$ of cardinality $\kappa$ have a common $\subm$ extension
$C$.

If this condition holds for models $A,B$ of any cardinality $\leq
\kappa$ ($<\kappa$) we write JEP$(\leq\kappa)$ (JEP$(<\kappa)$). In particular, $|A|$ and $|B|$ can be
different.

The \emph{full-joint embedding property} (full-JEP) is the JEP with
no restriction on the sizes of $A,B$.

The AEC $(\bK,\subm)$ has  the {\em amalgamation property}
at the infinite cardinal $\kappa$ (AP$(\kappa)$)  if for any three
 models $A,B,C$ each of cardinality $\kappa$ there are $\subm$-maps
 of $B$ and $C$ into an
extension $D$ which agree on $A$.  We use cardinal parameters as for
joint embedding.
\end{Def}

By Assumption \ref{BasicAssumption}, $\kappa$ is greater or equal to LS($\bK$) and without loss of generality we can assume that $C$ in the
definition of JEP and $D$ in the definition of AP, both have size $\kappa$.

The following easy consequences of the definitions show there are
some subtleties in the relation between joint embedding, maximal
models and arbitrarily large model.
\begin{Lem}\label{basicjep} Let $(\bK,\subm)$ be an AEC
\begin{enumerate}
\item If there are no maximal models then $\bK$ has arbitrarily
    large models.
    \item  If $(\bK,\subm)$ satisfies JEP$(\le\kappa)$ and has a
        model in power $\kappa$ then any model extends to one of
        size of $\kappa$; thus
\item If $(\bK,\subm)$ has full-JEP and  has arbitrarily large
    models then $\bK$ has no maximal models.
    \item If $\bK$ has two non-isomorphic maximal models in power
        $\kappa$, JEP$(\kappa)$ fails.
        \end{enumerate}
\end{Lem}

We will construct an AEC with at least two maximal models in a
cardinal $\lambda^+$.  Condition 4) says the most JEP possible is
JEP$(\leq \lambda)$.  Some of our examples will satisfy this. For
others we settle for JEP$(< \lambda)$.

We show that without the hypothesis of full-JEP the implication from
arbitrarily large models to no maximal models fails on various
countable sets of cardinals. There are some trivial examples for this
(see Corollary \ref{Cor:TrivialExample}), where one just takes
disjunctions of sentences (in disjoint vocabularies). However, the
disjunction of two sentences introduces properties that are clearly
artificial. In particular, one can find sentences with maximal models
in any countable set of cardinals by putting them in disjoint
vocabularies and taking a disjunction.  Such a class does not have
JEP in any cardinal. We eliminate these trivial examples using the
following definition. Moreover, our examples also satisfy JEP$(\le\lambda)$, for some infinite $\lambda$.

 \begin{Def} \label{Def:PureAEC}

 \begin{itemize}
 \item Let $\bK$ be a collection of $\tau$-structures and $\tau_1$ be a subset of $\tau$.
 Then $\bK_{\tau_1}$ is the subcollection of $\bK$ of
models where all symbols not in $\tau_1$ have the empty interpretation.

\item An AEC $(\bK,\subm)$ in a vocabulary $\tau$ is called {\em
 hybrid} if $\tau =\tau_1 \cup \tau_2$, $\bK = \bK_{\tau_1}\cup
 \bK_{\tau_2}$ and at least one of $\tau_1$, $\tau_2$ is not equal to $\tau$.

 If $\bK$ is not hybrid then it is \emph{pure}.
 \end{itemize}
 \end{Def}

 The most trivial example of a hybrid AEC is defined by the
 disjunction of sentences in disjoint vocabularies. The definition
 allows a more subtle version where the vocabularies can overlap but
 one of the classes forces some of the relations to be empty.  Lemma \ref{Lem:FullJEP} provides an example of how hybrid AEC that are not just
disjunctions of sentences in disjoint vocabularies  give trivial counterexamples.

Note that trivially JEP$(<\kappa)$ for all $\kappa$, or AP$(<\kappa)$
for all $\kappa$, imply full-JEP and full-AP respectively. On the contrary, Lemma \ref{Lem:FullJEP} proves that the assumption JEP$(<\kappa)$, for
all $\kappa$, can not be replaced by the assumption JEP$(\kappa)$, for all $\kappa$. In addition (see Corollary \ref{Cor:TrivialExample}) the
example in Lemma \ref{Lem:FullJEP} has AP$(\kappa)$, for all uncountable $\kappa$, but fails AP$(\aleph_0)$ and thus it fails AP$(<\kappa)$, for all
uncountable $\kappa$.
So there is a genuine distinction between these types of conditions. This is an important distinction since the definition of a good-$\kappa$
frame requires the weaker AP$(\kappa)$ and not AP$(<\kappa)$.

\begin{Lem}\label{Lem:FullJEP} The full-Joint Embedding Property is \emph{not} equivalent to the conjunction of JEP$(\kappa)$,
for all infinite $\kappa$.
\begin{proof}
Let $\tau$ be the vocabulary $\{V,U,E,<\}$, where the $V,U$ are unary
predicate symbols, and $E,<$ are binary predicate symbols. Consider
$\phi_1$ to be the conjunction of $\tau$-sentences asserting:
\begin{enumerate}
 \item $V,U$ partition the universe;
\item $(U,<)$ is well-ordered in order type $\omega$; and
\item $E$ defines a bijection from $V$ onto $U$.
\end{enumerate}

Let $\phi_2$ be the conjunction of
\begin{enumerate}
 \item $U$ is empty, $V$ is infinite and equals the universe; and
\item $<,E$ are empty
\end{enumerate}

Let $\bK$ be the collection of models of  $\phi_1 \vee \phi_2$, and
let $\subm$ be the substructure relation.

If $\M_1,\M_2$ are two models of $\phi_2$ of the same cardinality
$\kappa$, then they are isomorphic so can be jointly embedded. If
$\M_1,\M_2$ are two (necessarily countable) models of $\phi_1$, then
$\M_1$ and $\M_2$ are isomorphic. If $\M_1$ is a model of $\phi_1$
and $\M_2$ is a countable model of $\phi_2$, then $\M_2$ can be
embedded in $\M_1$. So, for all infinite $\kappa$, JEP$(\kappa)$
holds.

However, if $\M_1$ is a countable model of $\phi_1$ and $\M_2$ is an
uncountable model of $\phi_2$,
 $\M_1,\M_2$ have no common extension   in
$\bK$ (as $\M_1$ is maximal in $\bK$). So, full-JEP fails.
\end{proof}
\end{Lem}

\begin{Cor} \label{Cor:TrivialExample} The AEC $(\bK,\subm)$ defined in Lemma \ref{Lem:FullJEP}
is a hybrid AEC that satisfies JEP$(\kappa)$ for all infinite $\kappa$, has maximal models
in $\aleph_0$, and has arbitrarily large models, but fails JEP$(\le\aleph_1)$. Moreover, it satisfies AP$(\kappa)$, for all uncountable
$\kappa$, but it fails AP$(\aleph_0)$ and AP$(\le\aleph_1)$.
\end{Cor}

\begin{question} Is there a `pure' example (according to Definition \ref{Def:PureAEC}) to illustrate the distinction between full-JEP and
JEP$(\kappa)$ for all $\kappa$? \end{question}

We contrast this result with a  less complicated version of a result
of Shelah (Theorem 2.8 of \cite{Sh88}) which was originally stated
without proof.\footnote{The proof is sketched on page 134 of
\cite{Shaecbook}.  A weaker form of this result was reproved in
\cite{BLS} (Theorem 3.4), inadvertently without citation. The result
also occurs in Rami Grossberg's master's thesis.}.

%We include the proof for completeness.

\begin{Fact}\label{AP} If an AEC has AP$(\kappa)$ for every $\kappa$, then it
has the (full-) Amalgamation Property.
\end{Fact}

An easy variation of the proof shows:

\begin{Cor} If $\lambda<\kappa$ and an AEC satisfies JEP$(\lambda)$ and AP$(\le\kappa)$, then it also satisfies JEP$(\kappa)$ and even
JEP$(\le\kappa)$.
\end{Cor}

Thus, for the distinction made in Lemma \ref{Lem:FullJEP} between JEP$(\le\aleph_1)$ and the conjunction of JEP$(\aleph_0)$ and JEP$(\aleph_1)$,  it
is imperative for AP$(\le\aleph_1)$ to fail.

We can modify the example of  Lemma \ref{Lem:FullJEP} to allow
$(U,<)$ to be an infinite well-order of order type $\le\kappa$ (with
strong substructure as end-extension), for some cardinal $\kappa$.
The resulting AEC will satisfy JEP$(\lambda)$, for all infinite
$\lambda$, and even JEP$(\le\kappa)$, but fail JEP$(\le\kappa^+)$.

\begin{Cor} For all infinite $\kappa$,
\begin{enumerate}
 \item JEP$(\le\kappa^+)$ is \emph{not} equivalent to the conjunction of JEP$(\lambda)$, $\lambda\le\kappa^+$.
 \item JEP$(\le\kappa^+)$ is \emph{not} equivalent to the conjunction of JEP$(\le\kappa)$ and JEP$(\kappa^+)$.
\end{enumerate}
\end{Cor}

We will see with more difficulty below that there are pure AEC which exhibit the behavior of Corollary \ref{Cor:TrivialExample}, i.e.
they have both maximal models and arbitrarily large models.

\section{Basic combinatorics}\label{section:intro}
\numberwithin{Thm}{section}

In this section we set up a first-order template of bipartite graphs on
sets $A,B$ with colors from $C$.  We introduce the requirement that
there is no monochromatic $K_{2,2}$ subgraph (a complete bipartite graph on
points $a_1, a_2 \in A, b_1, b_2 \in B$ with all edges the same
color).  Then we show restrictions on the cardinality of $A$ and $B$
that are imposed by restrictions on the number of colors. In later
sections, we will impose the restrictions on $|C|$ by characterizing
them by sentences of $\lomegaone$ in the following sense.

\begin{Def} \label{Def:CharacterizableCard} An $\lomegaone$-sentence $\phi$ \emph{characterizes} an infinite cardinal $\kappa$,
 if $\phi$ has models in all cardinalities $\le\kappa$, but no model of size $\kappa^+$. In this
 case we say that the cardinal $\kappa$ is \emph{characterizable}.
\end{Def}

Since the Hanf number for $\lomegaone$-sentences is $\beth_{\omega_1}$, it follows that all characterizable
cardinals are strictly less than
$\beth_{\omega_1}$.

\begin{Not}\label{sigma1def} Let $\tau_0=\{A,B,C,E\}$ where $A,B,C$ are unary
predicates and $E$ is a ternary relation. Let $\sigma_0$ be the
conjunction of the following statements:

\begin{itemize}
\item $A,B,C$ are non-empty and partition the universe.
\item $E\subset A\times B\times C$ defines a total function
from $A\times B$ into $C$.
\end{itemize}

As a notation, let $F(a,b)$ the unique value $c$ such that
$E(a,b,c)$ holds.
$A$ and $B$ should be regarded as the two sides of a bipartite
graph and $C$ as the set of edge-labels. $E$ assigns a unique
label to any pair from $A\times B$.

Let $\sigma_1$ be the conjunction of $\sigma_0$ and

\begin{itemize}
\item[($*$)] for all distinct $a_1,a_2$ in $A$ and $b_1,b_2$ in
    $B$, the four values $F(a_i,b_j)$ ($i,j\in\{1,2\}$) are
    not all identical.
\end{itemize}
\end{Not}

We will also refer to $(*)$ as ``there are no monochromatic $K_{2,2}$ subgraphs''.
If $a_1,a_2$ are two distinct elements in $A$ and for some $b\in B$, $F(a_1,b)=F(a_2,b)$, we will say that there exists a \emph{monochromatic path} of
length $2$, or \emph{monochromatic $2$-path}, on $a_1,a_2$.

\begin{Lem}\label{Lem:BasicCombinatorics}
	In any model of $\sigma_1$, if $|A|>|C|^+$ then $|B|\leq|C|$. By
symmetry, the same is true if we switch the roles of $A$ and $B$.
\end{Lem}
\begin{proof}
	Toward a contradiction, assume that $|B|>|C|$. For any subset $D$
	of $B$ and any element $a\in A$, define $S(a,D)=\{F(a,d)|d\in D\}$
	(which is a subset of $C$).

	Now given any such $D$ of size $|C|$ and any $b\in B\setminus D$,
	we observe that for all but $|C|$ many elements $a\in A$,
	$S(a,D)\subsetneq S(a,D\cup\{b\})$. Indeed, if $S(a,D)=S(a,D\cup\{b\})$, then we have some $c\in C$ and $d_a\in D$ with $F(a,d_a)=F(a,b)=c$.
	If $a,a'$ are distinct elements of $A$ and $F(a,d_a)=F(a,b)=c=F(a',d_{a'})=F(a',b)$, then $d_a$ and $d_{a'}$ have to be distinct.
	If not, $a,a',d_a,b$ witness a violation to ($*$). So, for every color $c$, the set $\{a\in A|F(a,b)=c\}$ has size at most $|D|$. Since there
exist $|C|$ many colors, there are at most $|D|\cdot|C|=|C|$ many elements such that $S(a,D)=S(a,D\cup \{b\})$.
	
	Now let $(b_i|i<|C|^+)$ be a sequence of distinct elements in
	$B\setminus D$ and set $D_i=D\cup\{b_j|j<i\}$.
	For each $i<|C|^+$, let $A_i\subset A$ be the set of elements $a$
	such that $S(a,D_i)\subsetneq S(a,D_{i+1})$. Since
	$|A|>|C|^+$ and all $A\setminus A_i$ have size at most $|C|$,
	$A^*=\bigcap\limits_{i<|C|^+}A_i$ is non-empty (in fact its complement
	has size at most $|C|^+$). But for any element $a\in A^*$,
	$S(a, D_i)$ grows at each step $i<|C|^+$ which is impossible since
	$S(a, D_i)\subset C$ for all $i$.
	
\end{proof}

\section{Maximal models in many cardinalities}\label{incom}
\numberwithin{Thm}{subsection}

In this section we prove that one can
have interesting spectra of maximal models for AEC that are \emph{pure} (see Definition \ref{Def:PureAEC}). Specifically, we
construct sentences in $L_{\omega_1,\omega}$ that are not just
disjunctions of complete sentences. In Section~\ref{com}, we show limitations on getting such results for complete sentences compatible with
$\sigma_1$.

In \cite{Sh300} Shelah defines a {\em universal class} as one  that
is closed under substructure, union of chains, and isomorphism. He
remarks that by a result of Tarski, if the vocabulary is finite, then
such a class is axiomatized by a set of universal first order
sentences. This generalizes to:  If the vocabulary has cardinality
$\kappa$, the class is axiomatized in $L_{\kappa^+,\omega}$. For
simplicity here we use only countable vocabularies and
$L_{\omega_1,\omega}$-sentences.

\subsection{Maximal Models}\label{maxmod}
\medskip

Throughout this section $\M$ is a model of $\sigma_1$ of cardinality
$\kappa$. Since we discuss in this subsection only the construction
of extensions of a single model we are free to assume that $C
\subseteq \kappa$. We write $C^{\M} = C \subset \kappa$ to assert
that the interpretation of the predicate $C$ is a subset of $\kappa$
and use $|C|$ when we mean cardinality.

 In this section we build models $\M$ of $\sigma_1$ that are
$C$-maximal in the following sense.

\begin{Def} Let $|C| = \kappa$; $\M$ is a $C$-maximal model of $\sigma_1$ if $C^{\M} =
C$ and there is no proper extension of $\M$ to a model $\M'$ of
$\sigma_1$ with $C^{\M'} = C$.
\end{Def}

 In applications we require that we will expand
$\tau_0$ to a vocabulary $\tau' = \tau_0 \cup \tau_1$ and study
$\tau'$-models $M$ such that $M\restriction \tau_0 \models \sigma_1$
and $C^M \restriction \tau_1$    belongs to an AEC
$(\bK_0,\subKzero)$ for the vocabulary $\tau_1$, and $\bK_0$ has models in
cardinality $\kappa$ but no larger, and thus it has a maximal model in
$\kappa$.

\begin{Not}
	If $\M\models\sigma_1$, $|A^\M|=\kappa$ and $|B^\M|=\lambda$
	then we say that $\M$ is a $(\kappa,\lambda)$-model.
\end{Not}

In the following construction, we reverse the procedure of
Lemma~\ref{xxmodel} and build a model from a function on cardinals.

\begin{Lem}\label{xxmodel} For any $\kappa$, there is a $(\kappa^+,\kappa^+)$ model $\M\models\sigma_1$ such that $C^\M=\kappa$.
\end{Lem}

 \begin{proof} Let $A^\M$ and $B^\M$ be two copies of $\kappa^+$.

  Fix a function $F$ from $\kappa^+\times \kappa^+$ to $\kappa$ such  that   a)\footnote{This requirement is not
  needed now but is used in the proof of Lemma~\ref{lem:maximal1}.} for all
  $\alpha$,
   $F(\alpha,\alpha) = 0$ and b) for all $\alpha\in A$, $F(\alpha,\cdot)$ is a
one-to-one function when restricted to the set $\{\beta\in
B|\beta\le\alpha\}$. Symmetrically, demand that for all $\beta\in B$,
$F(\cdot,\beta)$ is a one-to-one function when restricted to the set
$\{\alpha\in A|\alpha\le\beta\}$. Both conditions are possible
because all initial segments have size $\kappa=|C^\M|$. Then define a
graph as in Notation~\ref{sigma1def} using this function.

 Towards contradiction, assume that there are distinct $\alpha_1,\alpha_2$ in $A$ and $\beta_1,\beta_2$ in $B$ with all
four values $F(\alpha_i,\beta_j)$ ($i,j\in\{1,2\}$) identical.
Without loss of generality assume that
$\max\{\alpha_1,\alpha_2,\beta_1,\beta_2\}=\alpha_1$. By the choice
of  $F$, $F(\alpha_1,\beta_1)$ must be
 different than $F(\alpha_1,\beta_2)$. Contradiction.
\end{proof}

\begin{Cor}\label{getmax}
	For any infinite cardinal $\kappa$, the class of all models
	$\N$ of $\sigma_1$ with $C^{\N} = C$ where $C=\kappa$ is
fixed, contains a $(\kappa^+,\kappa^+)$-model $\M$ that is $C$-maximal.
\end{Cor}

\begin{proof} By fixing $C=\kappa$ we have an $(\kappa^+,\kappa^+)$ model $\M$ by Lemma~\ref{xxmodel}.  But there is no
extension of $\M$ with either A or B of cardinality $>\kappa^+$ by
Lemma~\ref{Lem:BasicCombinatorics}.  The collection of extensions
$\N$  of $\M$ that satisfy  $\sigma_1$ with $C^{\N} = C$ is closed
under union since $\sigma$ is $\forall_1$. So some extension of $\M$
with cardinality $\kappa^+$ must have no extension.
\end{proof}

We can in fact give two explicit constructions that yield
nonisomorphic maximal models. The first proof uses Fodor's theorem,
which we state for the sake of completeness.

\begin{Fact}[Fodor] If $f$ is a regressive function on a stationary set $S\subset \kappa$, then there is a stationary set $T\subset S$ and some
$\gamma<\kappa$ such that
$f(\alpha) = \gamma$, for all $\alpha\in T$.
\end{Fact}

\begin{Lem}\label{lem:maximal1} If we modify the construction of Lemma \ref{xxmodel} to require that
\begin{center}
 \begin{minipage}{.85\textwidth} $(\dag)$ for all $\alpha\in A$, $\alpha\ge \kappa$,
 the function $F(\alpha,\cdot)$ restricted to  $\{\beta\in B|\beta<\alpha\}$ is onto $C-\{0\}$,
 \end{minipage}
 \end{center}
then we obtain a $C$-maximal model.
 \begin{proof}
Recall for each $\alpha$, $F(\alpha,\alpha) = 0$.  First note that if
we extend $B$ by a new point $b$, then there must exist some $i\in C$
and a stationary subset $S_i$ of $A$ such that all the edges between
$s\in S_i$ and $b$ are colored $i$. Without loss of generality assume
that $S_i\subset \kappa^+\setminus\kappa$.

Now, define a function $g$ from $\kappa^+\setminus\kappa$ to $\kappa^+$ by
\[g(\alpha)=\text{ least $\beta<\alpha$ such that $F(\alpha,\beta)=i$.}\]
By $(\dag)$, $g$ is well-defined on all $\kappa^+\setminus\kappa$, and by the definition, $g$ is regressive, i.e. $g(\alpha)<\alpha$.
By Fodor's
Theorem
  we get a stationary $T_i\subset S_i$ and a $\gamma_i$ such
  that for each $t \in T_i$, $F(t,\gamma_i) = i$. But this contradicts $(*)$ of Notation~\ref{sigma1def}.
 \end{proof}
\end{Lem}

\begin{Not}Consider the following condition $(\ddag)_A$.

\begin{center}
 \begin{minipage}{.85\textwidth} $(\ddag)_A$ For any pair $(a,a')\in A^2$  and for any color $c$, there exists some $b\in B$ such that $F(a,b)=F(a',b)=c$.
 \end{minipage}
 \end{center}

Similarly, define $(\ddag)_B$ by exchanging the role of $a$'s and
$b$'s in $(\ddag)_A$, and let:

\begin{center}
 \begin{minipage}{.85\textwidth} $(\ddag)$ is the conjunction of $(\ddag)_A$ and
$(\ddag)_B$.
 \end{minipage}
 \end{center}
\end{Not}

\begin{Lem}\label{lem:maximal2}
 If $\M$ is a ($\kappa^+,\kappa^+$)-model of $\sigma_1\wedge(\ddag)_A$,
 then $B^\M$ can not be extended, and symmetrically, $A^\M$ can not be extended from a ($\kappa^+,\kappa^+$)-models
 of $\sigma_1\wedge(\ddag)_B$.

 Thus, if $\M$ is a model of $\sigma_1\wedge (\ddag)$  with $C^{\M} =
\kappa$ and $C^\M=\kappa$, then $\M$ is
 $C$-maximal.
\begin{proof}
 Assume a model satisfies $(\ddag)_A$  with $|C^\M| =\kappa$ and $C \subseteq \kappa$. Towards contradiction,
 assume we can extend $B$ by one element,
 say $b$. Since there are $\kappa$ many colors
and $\kappa^+$ many
elements $a\in A$ to connect to $b$, there will be two elements $a_1,a_2\in A$ so that both
 edges $(a_1,b),(a_2,b)$ get the same color $c$. Then
$(\ddag)_A$ gives a  contradiction to $(*)$.
\end{proof}
\end{Lem}

\begin{Lem}\label{Lem:ddagModel} There exists a $(\kappa^+,\kappa^+)$-model of $\sigma_1$ that satisfies $(\ddag)$ and $C=\kappa$.
\begin{proof} Proceed as in the proof of Lemma \ref{xxmodel}. At every stage $\alpha$ choose either a pair $a_1,a_2<\alpha$ in $A$ or
a pair $b_1,b_2<\alpha$ in $B$, and some color $c\in C$. Organize the
construction so that every combination of a pair and a color appears
at exactly one stage. This is possible, since there are $\kappa^+$
stages and $\kappa^+$  such combinations.

Assume $(a_1,a_2)$ and $c$ are chosen at stage $\alpha$. If there is a $2$-path on $(a_1,a_2)$ colored by $c$, then do nothing more than what the
proof of Lemma \ref{xxmodel}
requires. If there is no such pair, require that the new edges $(a_1,\alpha)$ and $(a_2,\alpha)$ are both colored by $c$.
 This is a small violation
of the  requirement that $F(\cdot,\alpha)$ is $1$-$1$;  demand that
this is the only violation. Make the analogous choice when
$(b_1,b_2)$ and $c$ are chosen.

We claim that the resulting construction satisfies $(*)$ and
obviously satisfies $(\ddag)$. Towards contradiction, assume that
there are distinct $\alpha_1,\alpha_2$ in $A$ and $\beta_1,\beta_2$
in $B$ with all four values $F(\alpha_i,\beta_j)$ ($i,j\in\{1,2\}$)
equal to the same value $c$. Without loss of generality assume that
$\max\{\alpha_1,\alpha_2,\beta_1,\beta_2\}=\alpha_1$. Observe that
$F(\alpha_1,\beta_1)=F(\alpha_1,\beta_2)=c$ is possible only if the
pair $\beta_1,\beta_2$ and the color $c$ were chosen at stage
$\alpha_1$. Split into two cases:

Case 1. $\alpha_2>\beta_1,\beta_2$. Then the same observation (for
$\alpha_2$) proves that the same pair $\beta_1,\beta_2$ and the same
color $c$ that were chosen at stage $\alpha_1$ were also chosen at
stage $\alpha_2$. But it is impossible for the same combination of
pair and color to appear more than once. Contradiction.

Case 2. $\alpha_2\le\max\{\beta_1,\beta_2\}$. Then at stage
$\alpha_1$, there already exists a $2$-path on $(\beta_1,\beta_2)$
colored by $c$. The construction requires in this case that
$F(\alpha_1,\beta_1)$ be different than $F(\alpha_1,\beta_2)$ which
 again yields a contradiction. So, $(*)$ holds.
\end{proof}
\end{Lem}

The requirements of Lemma \ref{lem:maximal1} and
Lemma~\ref{lem:maximal2} are contradictory, so there are two
nonisomorphic $C$-maximal models of $\sigma_1$.

\begin{Cor}\label{cor:TwoNoniso} For all infinite cardinals $\kappa$,
there is a model $\M$ with $C^{\M} = \kappa$ that has two
non-isomorphic extensions that are $C$-maximal.
\end{Cor}

We can  vary the constructions and get still other maximal models;
these construction will used in the next section.

\begin{Cor}\label{cor:maximal1} If $A^{\M} =\kappa^+$, $A_0$ is a club in $\kappa^+$  with
$A_0\cap\kappa=\emptyset$, and $C^\M \subset \kappa$ then condition
$(\dag)$ in Lemma \ref{lem:maximal1} can be relaxed to the following
condition.
\begin{center}
 \begin{minipage}{.85\textwidth} $(\dag)_{A_0}$ For all $\alpha\in A_0$, the function $F(\alpha,\cdot)$ restricted to the set $\{\beta\in
B|\beta<\alpha\}$ is onto $C-\{0\}$.
 \end{minipage}
 \end{center}

and we still get that $\M$ is $C$-maximal.

\end{Cor}

\begin{Cor}\label{cor:maximal2} If $A^{\M} =\kappa^+$, $A_0$ is a subset of $A$ of size $\kappa^+$, and $C^\M \subset \kappa$, then
 $(\ddag)_A$ can be relaxed to
\begin{center}
 \begin{minipage}{.85\textwidth} $(\ddag)_{A_0}$ For any pair $(a,a')$, $a,a'\in A_0$, and for any color $c$,
  there exists some $b\in B$ such that
$F(a,b)=F(a',b)=c$.
 \end{minipage}
 \end{center}

and we still get that $\M$ is $C$-maximal.
\end{Cor}

Notice that while condition $(\ddag)$ can be expressed by a
first-order sentence in the same vocabulary as $\sigma_1$, this is
not the case for $(\dag)$ and $(\dag)_{A_0}$. The latter conditions
make use of the ordering $<$ that we used during the proof which is
not part of the vocabulary.

Corollary \ref{cor:maximal2} will be used to construct infinitely
many nonisomorphic maximal models of $\sigma_1$ in Section
\ref{sec:maximalmodel}.
The existence of maximal models is complemented by the following lemma.

\begin{Lem}\label{getmod}
	For any $\kappa$, there is a model $M\models\sigma_1$ with
	$|A|$ arbitrary large, $|B|\le\kappa$ and $|C|=\kappa$.
\end{Lem}
\begin{proof}
	Let $A$ be an arbitrary set, $B=\{b_\alpha|\alpha<\gamma\le\kappa\}$, and $C=\{c_\alpha|\alpha<\kappa\}$
	such that $A,B,C$ are pairwise disjoint. For any $a\in A$ and $b_\alpha\in B$, set $F(a,b_\alpha)=c_\alpha$.
	We cannot have a contradiction to ($*$) since each element in $B$
 is connected only by edges of a fixed color  and distinct elements
	in $B$ get distinct colors.
\end{proof}

\begin{Cor}\label{cor:LargerThanKappa}
	For any infinite cardinal $\kappa$, the class of all models
	of $\sigma_1$ with $|C|=\kappa$ has arbitrary large models.
	Moreover, in any model larger than $\kappa^+$, exactly one of
	$A$ or $B$ has to be no larger than $\kappa$.
\end{Cor}

\subsection{Failure for  Complete Sentences}\label{com}
\numberwithin{Thm}{subsection}

We show that our main combinatorial idea does not support the maximal
model spectra given above, if the $\lomegaone$-sentence is required
to be complete.  For this we need to formalize the consequences of
our two types of constructions of maximal models. The next lemma
proves that the models of $\sigma_1$ given in Lemma \ref{getmod} are
typical of $(\lambda,\kappa)$-models, where $\lambda\ge\kappa^+$. We
need one definition first.

\begin{Def} Let $\M= (A,B,C,E)$ be colored by $F$.  For $a\in A$, let $C_a=range(F(a,\cdot))$, and for $c\in C_a$ let
$$B_{a,c}=\{b\in B|F(a,b)=c\}.$$
\end{Def}

\begin{Lem} Let $\M$ be a $(\lambda,\kappa)$-model of $\sigma_1$, $\lambda\ge\kappa^+$, such that $|C^\M|=\kappa$. For all but $\kappa$ many $a\in A$
and for all $c\in C_a$, $|B_{a,c}|=1$.
\begin{proof} Assume otherwise, i.e. there are at least $\kappa^+$ many $a\in A$ such that there exists some $c_a\in C_a$ so that $|B_{a,c_a}|\ge 2
$. Call $A_0$ the set of these $a$'s. Since $A_0$ has size $\kappa^+$
and $C$ has size $\kappa$, we can restrict $A_0$ to some subset $A_1$
of size $\kappa^+$ such that $c_a=c$, for all $a\in A_1$. Then, for
each $a\in A_1$ choose a $2$ element subset $B'_{a,c}$ of $B_{a,c}$.
Since there are only $\kappa$ many $2$-element subsets of $\kappa$,
there exist $a_1,a_2\in A_1$, $B'_{a_1,c}=B'_{a_2,c}$. But $a_1,a_2$
witness that $(*)$ is violated.  Contradiction.
\end{proof}
\end{Lem}

Now we formalize this distinction.

\begin{Lem}\label{cor:OneColor} Let $\tau_1$ be the (first-order) statement: ``There exists some $a\in A$ so that for all $c\in C_a$,
$|B_{a,c}|=1$''. If $|C|=\kappa$ and $\lambda\ge\kappa^+$ , then any
$(\lambda,\kappa)$-model of $\sigma_1$ satisfies $\tau_1$, while
$\tau_1$ is obviously false in $(\kappa^+,\kappa^+)$-models.
\end{Lem}

\begin{Cor} There is no model $\M$ of size $\kappa^{++}$ such that $|C^\M|=\kappa$ and
$\M$ satisfies the conjunction $\sigma_1\wedge \neg\tau_1$.
\end{Cor}

Now we show this combinatorics will not give a complete sentence with
two maximal models.

\begin{Thm}\label{nocomplete}  For each $\kappa$ and each $\tau' \subseteq \tau_0$, there is no complete $\tau'$-sentence $\phi_\kappa$ such
 that (a) $\phi_\kappa$ allows at most $\kappa$ colors, (b) $\phi_\kappa$ is consistent
 with $\sigma_1$, (c) $\phi_\kappa$ has maximal models in some cardinal $\lambda\ge\kappa^+$ and
  (d) $\phi_\kappa$ has arbitrarily large models.
\begin{proof} By Lemma \ref{Lem:BasicCombinatorics}, if $\phi_\kappa$ has a model of
 cardinality $\lambda\ge\kappa^{++}$, this is a $(\lambda,\kappa)$-model. Then by Corollary \ref{cor:OneColor},
  $\phi_\kappa$ is consistent with $\tau_1$. In particular, $\phi_\kappa$ does not have any $(\kappa^+,\kappa^+)$- models.
  So, all models of $\phi_\kappa$ of size $\lambda\ge\kappa^+$ are $(\lambda,\kappa)$-models and by
 Corollary \ref{cor:NotMaximal}, any such model can not be maximal.
\end{proof}
\end{Thm}

We see that all sufficiently large models are extendible.

\begin{Cor}\label{cor:NotMaximal} If $\M$ is a $(\lambda,\kappa)$-model of $\sigma_1$, $\lambda\ge\kappa^+$, and
$|C^\M|=\kappa$, then the $A$-side of $\M$ is extendible, while keeping the $B$-side of $\M$ and $C$ the same.

In particular, $\M$ is not
maximal.
\begin{proof} By Corollary \ref{cor:OneColor}, $\tau_1$ holds and let $a$ be an element that witnesses $\tau_1$. Extend $A$ by adding a new element
$a'$ and letting
$F(a',b)=F(a,b)$, for all $b\in B$. It is immediate that $(*)$ holds in the new model.
\end{proof}
\end{Cor}

\begin{obs}  Before we move to the next section note that the requirement that
the requirement of $C$-maximality that appears in the results of this
section can be replaced by the requirement that $C$ is the universe
of a maximal model of an $\lomegaone(\tau')$-sentence $\phi$ for some
$\tau'$ disjoint from $\tau_0$ and $\phi$ characterizes $\kappa$ in
the sense of Definition \ref{Def:CharacterizableCard}. More
generally, we can require that $C$ belongs to an AEC with models in
cardinality $\kappa$, but no larger.
\end{obs}

\subsection{The Maximal Model Functor and JEP}\label{maxfn}

In this section we define a functor which takes us from an AEC
which has models of size $\kappa$ but no models in $\kappa^+$, to an AEC with a maximal model in $\kappa^+$
but arbitrarily large models.

{\em For the rest of the paper ${\hat \bK_0}$ is taken to depend on
$\bK_0$ as in the next definition.} We build the construction using
$\sigma_1$ from Notation~\ref{sigma1def}.

\begin{Def}\label{funcdef} Let $(\bkzero,\subKzero)$ be an AEC. The vocabulary
 of ${\hat \bkzero}$ is ${\hat \tau}_0 =\tau_0 \cup \tau_{{\mbox{\scriptsize ${\boldmath {K_0}}$}}}$. Let ${\hat \bkzero}$
be the collection of models of $\sigma_1$ with the color sort $C$ the
domain of a model in $\bkzero$. Define for $M,N\in {\hat \bkzero}$,
$M\subhat N$, if $M\subset N$ and $C^M\subKzero C^N$.
\end{Def}

\begin{Lem}\label{Lem:LS(K)} $({\hat \bK_0},\subhat)$ is an AEC with the same L\"{o}wenheim-Skolem number as $\bkzero$.
\begin{proof} Since $\sigma_1$ is a $\forall^0_1$-first-order sentence, ${\hat \bK_0}$ is closed under
direct limits. The coherence axiom is straightforward. So the only
issue is to check the L\"{o}wenheim-Skolem number. Let $M$ be a model
in ${\hat \bK_0}$ and $X$ be a subset of $M$. Find some $C_1\in
\bkzero$ such that $X\cap C^M\subset C_1$, and $|C_1|=|X\cap
C^M|+LS(\bkzero)$. Let $M_0=X\cup C_1$. In particular, $C^{M_0}=C_1$
and $M_0$ belongs to ${\hat \bK_0}$. Indeed, $M_0$ satisfies
$\sigma_1$, since any violations of $(*)$ in $M_0$ would be violations
of $(*)$ in $M$ too. Contradiction. Considering that
$|M_0|\le|X|+|C_1|\le|X|+|X|+LS(\bkzero)=|X|+LS(\bkzero)$, it follows
that $LS({\hat \bK_0})=LS(\bkzero)$.
\end{proof}
\end{Lem}

\begin{Thm}\label{Thm:UncountableJEPAP} Let $\kappa$ be an uncountable cardinal,
 $\bkzero$ and ${\hat \bK_0}$ be as in Definition~\ref{funcdef}, and suppose   $\bkzero$
has models in cardinality $\kappa$, but no larger.

Then ${\hat \bK_0}$ is an AEC that satisfies the following
\begin{enumerate}
\item If $\lambda\leq \kappa$ then $\bkzero$ satisfies
    JEP$(\leq\lambda)$ if and only if ${\hat \bK_0}$ satisfies
    JEP$(\leq\lambda)$. The equivalence  extends to
    JEP$(<\lambda)$ and JEP$(\lambda)$.
\item AP fails in all infinite cardinals;
\item ${\hat \bK_0}$ has at least 2 maximal models in $\kappa^+$
    and none in any $\lambda \neq\kappa^+$; moreover, ${\hat
    \bK_0}$ fails JEP$(\le\lambda)$, even JEP$(\lambda)$, for
    $\lambda \ge \kappa^+$.
\item ${\hat \bK_0}$ has arbitrarily large models; and
\item $LS({\hat \bK_0})=LS(\bkzero)$.
\end{enumerate}
Moreover, ${\hat \bK_0}$ is a pure AEC, in the sense of Definition
\ref{Def:PureAEC} if and only if $\bkzero$ is pure.
\begin{proof}  First observe that  since $\bkzero$ characterizes $\kappa$ it  must contain some maximal models in $\kappa$.

(1) Clearly if     ${\hat \bK_0}$ satisfies
    JEP$(\leq\lambda)$ then $\bkzero$ satisfies
    JEP$(\leq\lambda)$. For the converse, fix $\lambda \leq \kappa$ and suppose $\bkzero$
satisfies JEP$(\leq\lambda)$; we show ${\hat \bK_0}$ satisfies
    JEP$(\leq\lambda)$. The other two cases (JEP$(<\lambda)$, JEP$(\lambda)$) are similar. Let $\M_1=(A_1,B_1,C_1)$,
$\M_2=(A_2,B_2,C_2)$ be two models in ${\hat \bK_0}$ such that both
$M_1,M_2$ have size $\leq \lambda$. Use JEP on $\bkzero$ and
Lemma~\ref{basicjep}.2 to find a common extension $C$ of both
$C_1,C_2$ with cardinality at most $\lambda$. Then consider the
structure $(A_1\cup A_2, B_1\cup B_2, C)$.  By identifying $C_1$ and
$C_2$ with subsets of C, we can consider all existing edges as
$C$-colored. Then assign colors to new edges in a one-to-one way.
This is possible, since that there are no more than $\lambda$ many
edges and $\lambda$ many colors. Towards contradiction assume there
is a violation of $(*)$ witnessed by the edges $(l,l',r,r')$. If
there were three old edges among these then all four vertices would
be in $\Mscr_1$ or $\Mscr_2$. So there are two new edges, but the new
edges were colored, so there can not be two new edges among
(l,l',r,r') with the same color. Contradiction.

(2) For any $\lambda$, let $\M_i=(A_i,B_i,C_i)$, $i=1,2,3$, be three
models of $\hat \bK_0$ with  cardinality $\lambda$ such that
$\M_1\subset \M_2,\M_3$ and there exist distinct $a_0,a_1,a_2\in
A_1$, distinct $c,c',c''\in C_1$, $b_2\in B_2$, and $b_3\in B_3$ such
that $F(a_k,b_l)=c$, for $k=0,1$ and $l=2,3$, $F(a_2,b_2)=c'$, and
$F(a_2,b_3)=c''$. (Extensions $\M_2$ and $\M_3$ of any $\M_1$ must
exist).  Thus,  in the disjoint amalgam of $\M_2$ and $\M_3$,
$a_0,a_1,b_2,b_3$ witness a violation to $(*)$. But in any amalgam of
$M_1,M_2,M_3$, the images of $b_2$ and $b_3$ must be distinct, and
thus, $a_0,a_1,b_2,b_3$ witness a violation to $(*)$ in any amalgam.

(3) By Corollary \ref{cor:TwoNoniso} there exist  two maximal models
in $\kappa^+$; by Corollaries \ref{cor:LargerThanKappa} and
\ref{cor:NotMaximal} no model of size $>\kappa^+$ can be maximal. By
(1) no model of cardinality $\leq\kappa$ can be maximal.
 Since there are two maximal models in $\kappa^+$, ${\hat \bK_0}$ fails
JEP$(\kappa^+)$ and JEP$(\le\lambda)$ for $\lambda \ge \kappa^+$.
To see that ${\hat \bK_0}$ fails JEP$(\lambda)$ for $\lambda>\kappa^+$,
consider a model $\M_0$ of type $(\lambda,\kappa)$ and a model $\M_1$ of type $(\lambda,\kappa)$.
By Corollary \ref{cor:LargerThanKappa}, $\M_0$ and $\M_1$ can not be jointly embedded into a model
in ${\hat \bK_0}$.

(4) is established by Corollary \ref{cor:LargerThanKappa}. The proof
of (5) is from Lemma \ref{Lem:LS(K)}
\end{proof}
\end{Thm}

Suppose $\bkzero$ has models in cardinality $\kappa$, but no larger,
and $\bkzero$ satisfies JEP$(\leq \kappa)$.  It follows from Theorem
\ref{Thm:UncountableJEPAP} that ${\hat \bK_0}$ will satisfy
JEP$(\le\kappa)$ and have a maximal model in $\kappa^+$. This
condition on $\bkzero$ is very strong: there is a unique maximal
model in $\kappa$.  However, examples of this sort (e.g. the
well-orderings of order type at most $\omega_1$ under end-extension)
are well-known.

\begin{question} Is there a complete sentence of
$L_{\omega_1,\omega}$? that has more than one maximal
model?\end{question}

\subsection{$ \mathbf 2^{\kappa^+}$ Nonisomorphic Maximal Models}
\label{sec:maximalmodel}

In this section we prove that the AEC given by Theorem
\ref{Thm:UncountableJEPAP} actually has $2^{\kappa^+}$ many
nonisomorphic maximal models in $\kappa^+$.  We will build a family
of models of $\sigma_1$, each one starting with sets $A,B,C$, the
first two ordered as $\kappa^+$, $C$ ordered as $\kappa$, and with a
subset $C_0$ of $C$ that also has order type $\kappa$, and with $C
\setminus C_0$ has cardinality $\kappa$.

 We are building models of the AEC $\hat \bK_0$ with vocabulary
$\hat \tau_0$ using an input AEC $\bK_0$ with vocabulary $\tau_0$ to
control the cardinality of the color sort. The key step in the
construction is to add new relations to the vocabulary $\hat \bK_0$
and use them to construct many models (in the expanded vocabulary).
But then, we show these relations are definable in
$L_{\kappa,\omega}({\hat \tau}_0)$ and deduce many ${\hat
\tau}_0$-models.

The proof goes in two steps. At the first step we ``code'' a linear
order of order type $\kappa$ on $C_0$. At the second step
we make use of this linear order to ``code'' $\kappa^+$ many subsets
of $\kappa$ into $A$. By varying the construction we get
$2^{\kappa^+}$ many nonisomorphic maximal models.

Recall that there exists a monochromatic 2-path (based) on some
$a_1,a_2\in A$, if there exists some $b\in B$, such that both edges
$(a_1,b)$ and $(a_2,b)$ have the same color.

\medskip
{\bf Step I} Code order:

Let $C$ be the set of colors and assume $C=\kappa$. Extend the
vocabulary ${\hat \tau}_0$ to ${\hat \tau}_1$ by including a unary
symbol $C_0$ and a binary symbol $<$.  $C_0$ will be a subset of $C$
and $<$ will be a linear order on $C_0$ of order type $\kappa$. The
goal is to build a model as in Lemma \ref{Lem:ddagModel}, but this
time certain 2-paths are disallowed. In particular, for all
$\alpha<\kappa$ there exist two elements $l^\alpha_1,l^\alpha_2\in A$
and the 2-paths on $l^\alpha_1,l^\alpha_2\in A$ can not use any of
the colors $\{\beta|\beta\le\alpha\}$. Any other color is allowed.
The resulting model is maximal, as seen by Corollary
\ref{cor:StepImax}.

\begin{Lem}\label{Lem:Step1} There is a ${\hat \tau}_1$-model $\M$ that satisfies all the following conditions.
\begin{enumerate}\setcounter{enumi}{-1}
\item $C^\M \restriction \tau_{\bK_0} \in \bK_0$.
\item $\M \restriction \tau_0$ is a $(\kappa^+,\kappa^+)$-model
    of $\sigma_1$ and $C=\kappa$.
\item $C_0$ is a subset of $C$ such that $|C_0|=|C\setminus C_0|=\kappa$ and $<$ is an order on $C_0$ of order type $\kappa$. We may refer to the
elements of
$C_0$ using ordinals $<\kappa$.
\item $<$ is void outside $C_0$.
 \item For every $\alpha\in C_0$, there exist two elements $l^\alpha_1,l^\alpha_2\in A$ such
  that there exists a 2-path on $l^\alpha_1,l^\alpha_2$ colored by $c$ if and only if $c>\alpha$ or $c \in C\setminus C_0$.
 \item For distinct $\alpha,\alpha'\in C_0$, the elements $l^\alpha_1,l^\alpha_2, l^{\alpha'}_1,l^{\alpha'}_2$ are all distinct.
  \item For every pair $(a_1,a_2)$ in $A$ and for all $c\in C$, there exists a 2-path on $a_1,a_2$ colored by $c$, unless it is forbidden by clause $(4)$.
\end{enumerate}

\begin{proof} We now construct the model. Let $A,B,C, C_0$ be as in the first paragraph of
Section~\ref{sec:maximalmodel} and order  $C_0$ by $<$ so that the
requirements of clauses $(2)$ and $(3)$ are met. For every
$\alpha<\kappa$, select two elements $l^\alpha_1,l^\alpha_2\in
\kappa$ so that $\alpha\neq\alpha'$ implies
$\{l^\alpha_1,l^\alpha_2\}\cap\{l^{\alpha'}_1,l^{\alpha'}_2\}=\emptyset$.
The rest of the proof is similar to the proof of Lemma
\ref{Lem:ddagModel}, the only difference is that for every
$\alpha<\kappa$, the pair $l^\alpha_1,l^\alpha_2$ given by clause
$(4)$ do not have a 2-path with any color $c\le\alpha$. The rest of
the argument remains the same.
\end{proof}
\end{Lem}

A priori, the conditions in Lemma~\ref{Lem:Step1} are not ${\hat
\tau}_0$-invariant. We show in Lemma~\ref{Lem:LoneIso} that $C_0$ and
$<$ are definable in $L_{\kappa,\omega}({\hat \tau}_0)$ so they are.

\begin{Cor}\label{cor:StepImax} The models that satisfy the requirements of Lemma \ref{Lem:Step1} are
$C$-maximal.
\begin{proof} Since the set of all $l^\alpha_1,l^\alpha_2$ has size $\kappa$, the result follows from Corollary \ref{cor:maximal2}.
 \end{proof}
\end{Cor}

\begin{Lem}\label{Lem:LoneIso} Let $\M_1$ and $\M_2$ be two ${\hat \tau}_1$-models
that satisfy the conditions of Lemma \ref{Lem:Step1} and let
$\M_1|_{{\hat \tau}_0},\M_2|_{{\hat \tau}_0}$ be their reducts to
vocabulary ${\hat \tau}_0$. Then any isomorphism $i$ between
$\M_1|_{{\hat \tau}_0}$ and $\M_2|_{{\hat \tau}_0}$ is also an
isomorphism of $\M_1,\M_2$ (as ${\hat \tau}_1$-structures).

We will refer to this property as ``the ${\hat \tau}_0$-isomorphisms
respect $C_0,<$''.
\begin{proof} We claim that both $C_0,<$ are definable in the original structure $\Mscr$ by a
sentence of an appropriate infinitary language in vocabulary ${\hat
\tau}_0$, and therefore, preserved by ${\hat \tau}_0$-isomorphisms.

First, $C_0(x)$ is defined by ``there exists a pair
$(l^{\alpha}_1,l^{\alpha}_2)\in A$ such that there is no 2-path on
$l^\alpha_1,l^\alpha_2$ colored by $x$''.

Second, for each ordinal $\alpha<\kappa$, let $\alpha^{\M}$ denote the $\alpha^{th}$ element of the order $<^{\M}$.
 Since $<$ has order type $\kappa$, the specification makes sense. We prove by induction on $\alpha<\kappa$
 that $\alpha^{\M}$ is defined by a formula $\phi_\alpha(x)$ in
 $L_{{\kappa},\omega}$ in the vocabulary ${\hat \tau}_0$.

 $\phi_0(x)$: There exist two points $a,a'$ with no 2-path colored
 $x$. But for every other color $c\neq x$, there is a $c$-colored 2-path on $a,a'$.

 $\phi_{\alpha}(x)$: There exist two points $a,a'$ with no 2-path
 colored by $x$ or by any color $y$ satisfying $\bigvee_{\beta <\alpha}
 \phi_\beta(y)$. But for every other color $c\neq x$ and $\bigwedge_{\beta<\alpha} \neg\phi_\beta(c)$, there is a $c$-colored 2-path on $a,a'$.

 Now $<$ is defined by a $L_{\kappa,\omega}$-formula in the vocabulary ${\hat \tau}_0$.

 $$ x < y  \mbox{ \rm if and only if \ } \bigvee_{\alpha < \beta <
 \kappa} \phi_{\alpha}(x) \wedge \phi_{\beta}(y).$$

Since  each $\phi_\alpha(x)$ is a formula in vocabulary ${\hat
\tau}_0$, this proves the result.
\end{proof}
\end{Lem}

It also follows by a similar argument that the elements
$l^\alpha_1,l^\alpha_2$ are definable by a formula in
$L_{\kappa,\omega}$ in the vocabulary ${\hat \tau}_0$. Consider the
formula $\phi(x,y)$: ``there exists a 2-path on $x,y$ colored by $c$
if and only if $\neg \bigvee_{\beta\le\alpha} \phi_\beta(c)$. ''. By
clauses $(4)$ and $(6)$ of Lemma \ref{Lem:Step1}, $\phi(x,y)$ holds
if and only if $\{x,y\}=\{l^\alpha_1,l^\alpha_2\}$. So, any ${\hat
\tau}_0$-isomorphism must preserve the two-element subset
$\{l^\alpha_1,l^\alpha_2\}$, for all $\alpha<\kappa$.

\medskip { \bf Step II} Code subsets:

Recall that ${\hat \tau}_1={\hat \tau}_0\cup\{C_0,<\}$ and extend
${\hat \tau}_1$ to ${\hat \tau}_2$ by including a new binary symbol
$S$. $S$ will be a binary relation that codes subsets
$A_0=\{m_\alpha|\alpha<\kappa^+\}$ of $A$ by elements of  $C_0$.

We
also require that the set $\{l^\alpha_i|\alpha<\kappa,i=1,2\}$ from
Step I and the set $A_0$ from Step II are disjoint.  Using $S$ we can
assign to each $m_\alpha\in A_0$ a distinct subset $S_\alpha$ of
$C_0$. The goal is to build a model that satisfies all the
restrictions from Step I, plus  more 2-paths are forbidden. In
particular, for each $\alpha<\kappa^+$ the 2-paths based on
$m_0,m_\alpha$ can not use any of the colors in $S_\alpha$. Every
other color is allowed. Once again, the resulting model is maximal
(see Corollary \ref{cor:StepIImax}).

We again add predicates, this time to code models, and then prove
they are $L_{\kappa,\omega}({\hat \tau}_0)$ definable.

\begin{Lem}\label{Lem:Step2} There is an ${\hat \tau}_2$-model $\N$ that satisfies all the following conditions.
\begin{enumerate}
\item Clauses $(1)-(5)$ from Lemma \ref{Lem:Step1} hold.
\item There is a set $A_0 =\{ m_\alpha:{\alpha<\kappa^+}\}\subset A$ such that $|A\setminus A_0 |=\kappa^+$ and $A_0$ is disjoint from $\{l^\alpha_i|\alpha<\kappa, i=1,2\}$.
\item $S(x,y)$ is a binary relation on $A_0 \times C_0$. Denote the set $\{y\in C_0|S(m_\alpha,y)\}$ by $S_\alpha$.
\item  The $S_\alpha$'s are distinct subsets of $C_0$. For all $\alpha$, $|S_\alpha|=|C_0\setminus S_\alpha|=\kappa$, and $0$ does not belong to any     $S_\alpha$.
\item For all $0<\alpha<\kappa^+$, there exists a 2-path on $m_0,m_\alpha$ colored by $c$ if and only if $c\in S_\alpha$.
\item For all $a_1,a_2\in A$ and for all $c$, there exists a 2-path on $a_1,a_2$ colored by $c$, unless it is forbidden by clause $(5)$ of this Lemma
or by clause
$(4)$ of Lemma\ref{Lem:Step1}.
 \end{enumerate}

\begin{proof} The proof is similar to the proof of Lemma \ref{Lem:Step1} and is left to the reader.
\end{proof}
\end{Lem}

\begin{Cor}\label{cor:StepIImax} The models that satisfy the requirements of Lemma \ref{Lem:Step2} are $C$-maximal.
\begin{proof} Since $|A\setminus (A_0\cup \{l^\alpha_i|\alpha<\kappa, i=1,2\})|=\kappa^+$, the result follows from Corollary \ref{cor:maximal2}.
 \end{proof}
\end{Cor}

\begin{Lem}\label{Lem:LtwoIso} Let $\N_1$ and $\N_2$ be two ${\hat \tau}_2$-models that satisfy the conditions of Lemma \ref{Lem:Step2} and let
$\N_1|_{{\hat \tau}_0},\N_2|_{{\hat \tau}_0}$ be their reducts to
vocabulary ${\hat \tau}_0$. Then any isomorphism $i$ between
$\N_1|_{{\hat \tau}_0}$ and $\N_2|_{{\hat \tau}_0}$ is also an
isomorphism of $\N_1,\N_2$ (as ${\hat \tau}_2$-structures).

\begin{proof} From Step I we know that $C_0$ and $<$ are definable by $L_{\kappa,\omega}({\hat \tau}_0)$-formulas.
We prove that the same is true for the set $A_0
=\{m_\alpha|\alpha<\kappa^+\}$ and the sets $S_\alpha$,
$\alpha<\kappa^+$. The element $m_0$ is defined by the following
$L_{\kappa,\omega}({\hat \tau}_0)$-formula $\psi_0(x)$.

$\psi_0(x)$: there exist two distinct elements $m_1,m_2\in A$ and two distinct colors $c_1,c_2\in C_0$ and
there is no 2-path based on $x,m_1$ colored
by $c_1$,
and there is no 2-path based on $x,m_2$ colored by $c_2$.

We now show the set $\{m_\alpha|0<\alpha<\kappa^+\}$ is defined by
the formula $\psi_1(x)$.

$\psi_1(x)$: there exists some $y\neq x$ such that $\psi_0(y)$, i.e. $y$ equals $m_0$, and there exists a color $c\in
C_0$ and there is no 2-path based on $y,x$ colored by $c$.

Then $\psi_1(x)$ holds if and only if $x$ belongs to $\{m_\alpha|0<\alpha<\kappa^+\}$. Note that $\psi_1$ defines the whole set
$\{m_\alpha|0<\alpha<\kappa^+\}$,
but not the order of the $m_\alpha$'s in this set. Nevertheless, for every $\alpha<\kappa^+$, the set $S_\alpha$ is definable by the following formula
$\psi_2$ which uses $m_\alpha$ as a parameter; $\psi_2$ is a reformulation of clause $(5)$ from Lemma \ref{Lem:Step2}.

$\psi_2(x,m_\alpha)$: there exists some y such that $\psi_0(y)$ and there exists a 2-path on $y,m_\alpha$ colored by $x$.

Then $\psi_2(x,m_\alpha)$ holds if and only if $x\in S_\alpha$.

Since all these sentences are in vocabulary ${\hat \tau}_0$, this
finishes the proof.
\end{proof}
\end{Lem}

Now fix $Y$ to be some subset of $\kappa^+$ and vary the construction
of Lemma \ref{Lem:Step2} so that for each $0<\alpha<\kappa^+$, $0\in
S_\alpha$ if and only if $\alpha\in Y$. Call  the corresponding
${\hat \tau}_2$-structure $\N_Y$. If $Y_1,Y_2$ are two distinct
subsets, then $\N_{Y_1}$ and $\N_{Y_2}$ are easily seen to be
nonisomorphic as ${\hat \tau}_2$-structures. By Lemma
\ref{Lem:LtwoIso}, their ${\hat \tau}_0$-reducts are also
nonisomorphic, which proves the following.

\begin{Thm}\label{Thm:ManyMaxModels} If $\bkzero$ is an AEC that has models in cardinality $\kappa$ but no larger, then ${\hat \bK_0}$ from Theorem
\ref{Thm:UncountableJEPAP} has $2^{\kappa^+}$-many nonisomorphic maximal models of type $(\kappa^+,\kappa^+)$.
 \end{Thm}

In the next section we give three applications of Theorem \ref{Thm:ManyMaxModels}.

\subsection{Maximal Models in Many
Cardinalities}\label{sec:ManyCardinalities}

If $\kappa<\lambda<\beth_{\omegaone}$ are two characterizable
cardinals (Definition \ref{Def:CharacterizableCard}) and
$\bK_{\kappa}$, $\bK_{\lambda}$ the corresponding AEC (in disjoint
vocabularies), then the  union\footnote{By the union of AEC's with
disjoint vocabularies we mean the collection of structures in the
union of the vocabularies, where the obvious symbols have the empty
interpretation, and  one model is a strong substructure of another if
the same is true for their reducts to the vocabulary where the structures are non-trivial.} with an AEC with arbitrarily large models is an AEC (
with strong substructure being in the appropriate vocabulary) with
maximal models in $\kappa$ and $\lambda$
 and arbitrarily large models. However, the union is a hybrid AEC which fails JEP in all cardinals.

 If $<\lambda_i|i\le\alpha<\aleph_1>$ is a strictly increasing sequence
of characterizable cardinals (Definition
\ref{Def:CharacterizableCard}), we provide an example of a pure
(Definition~\ref{Def:PureAEC}) AEC with maximal models in
cardinalities $<\lambda_i^+|i\le\alpha<\aleph_1>$, arbitrarily large
models, and JEP$(<\lambda_0)$ holds.

For any triple $(A,B,C)$ there is a first order sentence $\sigma_1$
asserting that $A,B$ form a bipartite graph with $C$ many colors that
contains no monochromatic $ K_{2,2}$ subgraph (see property $(*)$).
The structures constructed below will contain many substructures
satisfying this requirement.  Rather than cluttering the paper with a
careful description of the formal sentence (with different ternary
relations for each colored graph) we will just assert where
$\sigma_1$ holds.

We begin with the case of two cardinals.

\begin{Lem}\label{2max} Let $\kappa<\lambda$ and let $(\bkzero^k,\prec_k)$ be an AEC in vocabulary $\tau^k$ with models in $\kappa$ but no higher,
and let $(\bkzero^\ell,\prec_\ell)$ be an AEC in vocabulary
$\tau^\ell$ with models in $\lambda$ but no higher. If both
$(\bkzero^k,\prec_k)$ and $(\bkzero^\ell,\prec_\ell)$ satisfy
JEP$(<\kappa)$, then there is an AEC $(\bK^*,\prec_{\bK^*})$ which
\begin{enumerate}
\item satisfies JEP$(<\kappa)$;
\item fails AP in all infinite cardinals;
\item has $2^{\kappa^+}$ non-isomorphic maximal models in
    $\kappa^+$, $2^{\lambda^+}$  non-isomorphic maximal models in
    $\lambda^+$, but no maximal models in any other cardinality,
    while JEP fails in all $\lambda \geq \kappa$;
\item has arbitrarily large models; and
\item $LS(\bK^*)=\max\{LS(\bkzero^k),LS(\bkzero^\ell)\}$.
\end{enumerate}
If both $(\bkzero^k,\prec_k)$ and $(\bkzero^\ell,\prec_\ell)$ are
pure, then $(\bK^*,\prec_{\bK^*})$ is pure. If both
$(\bkzero^k,\prec_k)$ and $(\bkzero^\ell,\prec_\ell)$ are definable
by an $\lomegaone$-sentence, then the same is true for
$(\bK^*,\prec_{\bK^*})$.

\begin{proof} Let $\bK^*$ be the AEC defined by the following construction.
  The construction contains 4 bipartite graphs entangled together.
  Recall that a bipartite graph is a $\tau_0$ structure and
  $\sigma_1$ is a $\tau_0$ sentence.
\begin{enumerate}[a)]
\item $A_1,A_2,A_3, C_1,C_2$ are non-empty and partition the universe.
\item The structures $(A_1,A_2,C_1)$, $(A_1,A_3,C_1)$,
    $(A_1,C_2,C_1)$, and $(A_2,A_3,C_2)$ are colored bipartite
    graphs satisfying $\sigma_1$.
\item $C_1$ is a model in $\bkzero^k$ and $C_2$ is a model in $\bkzero^\ell$. In particular $|C_1|\le\kappa$ and
    $|C_2|\le\lambda$
\end{enumerate}

Define for $\M,\N\in \bK^*$, $\M\prec_{\bK^*}\N$, if $\M\subset \N$
with respect\footnote{We abuse notation here; depending on the exact
location the colored graph will be with respect to a different
ternary relation; but we will think of it as a structure modeling the
appropriate translation of $\sigma_1$.} to $\tau_0$, $C_1^{\M}\prec_k
C_1^{\N}$ and $C_2^{\M}\prec_\ell C_2^{\N}$.

(1) Fix $\chi < \kappa$ and let
$\M_1=(A^1_1,A^1_2,A^1_3,C^1_1,C^1_2)$,
$\M_2=(A^2_1,A^2_2,A^2_3,C^2_1,C^2_2)$ be two models in $\bK^*$ such that both $M_1,M_2$ have size $\le	 \chi$.

Use JEP$(<\kappa)$ on $\bkzero^\ell$ to extend the
$\tau_\ell$-structures $C^1_2,C^2_2$ to a common structure
$\check{C}_2$ which has cardinality $\chi$. Use the argument of
Theorem~\ref{Thm:UncountableJEPAP}.1 to find a common embedding of
$(A^1_2,A^1_3,C^1_2)$ and $(A^2_2,A^2_3,C^2_2)$ with domain
$(A^1_2\cup  A^2_2,A^1_3\cup A^2_3,\check{C}_2)$ with cardinality
$\chi$.  Note that the proof of Theorem~\ref{Thm:UncountableJEPAP}.1
does not add any vertices to the graph. Then use JEP$(<\kappa)$ in
$\bK_0^k$ to find a common extension $\check{C}_1$ of $C^1_1$ and
$C^2_1$ of cardinality $\chi$. Now consider the structures
$(A^1_1,A^1_2,C^1_1)$, $(A^1_1,A^1_3,C^1_1)$,
$(A^1_1,\check{C}_2,C^1_1)$ and $(A^2_1,A^2_2,C^2_1)$,
$(A^2_1,A^2_3,C^2_1)$, $(A^2_1,\check{C}_2,C^2_1)$.  Apply the
argument of Theorem~\ref{Thm:UncountableJEPAP}.1 again several times
to  find an $\bK^*$ extension of all these models with domain
$(A^1_1\cup A^2_1, A^1_2 \cup A^2_2, A^1_3\cup A^2_3,
\check{C}_1,\check{C}_2)$. Exactly as in
Theorem~\ref{Thm:UncountableJEPAP}.1 we verify this structure is in
$\bK^*$.

(2)The proof for AP follows as in Theorem \ref{Thm:UncountableJEPAP}.

(3) and (4) Assume that $C_1$ has size $\kappa$. By  Lemma
\ref{Lem:BasicCombinatorics}, if $A_1$ has size $\kappa^+$, then
$A_2,A_3,C_2$ have size $\leq \kappa^+$ and by Theorem
\ref{Thm:ManyMaxModels} there are $2^{\kappa^+}$ many non-isomorphic
maximal models in $\kappa^+$. If $A_1$ has size $>\kappa^+$, then
$A_2,A_3,C_2$ have size $\leq \kappa$, and notice that the size of $A_1$ can be arbitrarily
large. If $A_1$ has size $\kappa$, then the sizes of
$A_2,A_3,C_2$ can be greater than $\kappa^+$.

Repeating the same argument, assume that $C_1$ and $A_1$ have size
$\kappa$, and $C_2$ has size $\lambda$. If $A_2$ has size
$\lambda^+$, then $A_3$ has size $\leq\lambda^+$ and by Theorem
\ref{Thm:ManyMaxModels} again, there are $2^{\lambda^+}$ many
nonisomorphic maximal models in $\lambda^+$. If $A_2$ (or $A_3$) has
size $\lambda$, then $A_3$ (respectively $A_2$) can have any size and
we get arbitrarily large models.

The failure of JEP in $\lambda\geq \kappa$ now fails as
Theorem~\ref{Thm:UncountableJEPAP}.

(5) The argument is similar to the proof of Lemma \ref{Lem:LS(K)}.
Let $M$ be a model in $\bK^*$ and $X$ be a subset of $M$. Find some
$\check{C}_1\in \bkzero^k$ such that $X\cap C_1^M\subset \check{C}_1$
and $|\check{C}_1|=|X\cap C_1^M|+LS(\bkzero^k)$. Then find some
$\check{C}_2\in\bkzero^\ell$ such that $X\cap C_2^M\subset
\check{C}_2$ and $|\check{C}_2|=|X\cap C_2^M|+LS(\bkzero^\ell)$. Let
$M_0=X\cup \check{C}_1\cup \check{C}_2$. Then $M_0$ belongs to ${\hat
\bK_0}$. Indeed, $M_0$ satisfies $\sigma_1$, since any violations of
$(*)$ in $M_0$ would be violations of $(*)$ in $M$ too.
Contradiction. Considering that
$|M_0|\le|X|+|\check{C}_1|+|\check{C}_2|\le|X|+LS(\bkzero^k)+LS(\bkzero^\ell)=|X|+\max\{LS(\bkzero^k),LS(\bkzero^\ell)\}$,
it follows that $LS(\bK^*)=\max\{LS(\bkzero^k),LS(\bkzero^\ell)\}$.

Finally observe that the conjunction of (a)-(c) is expressible by an $\lomegaone$-sentence if and only if membership in both $\bkzero^k$
and $\bkzero^\ell$ is expressible by an $\lomegaone$-sentence.
\end{proof}
\end{Lem}

We sketch a minor variant in the  argument to extend this to
infinitely many cardinals.

\begin{Thm}\label{manymax} Let $\langle \lambda_i: i\le \alpha \rangle$ be a strictly increasing
 sequence of cardinals. Assume that for each $i\le\alpha$, there exists an AEC $(\bkzero^i,\prec_{\bkzero^i})$ with models in $\lambda_i$ but no higher.
 Then if all $(\bkzero^i,\prec_{\bkzero^i})$ satisfy JEP$(<\lambda_0)$,
there is an AEC $(\bK^*,\prec_{\bK^*})$ which
\begin{enumerate}
\item satisfies JEP$(< \lambda_0)$, while JEP fails for all larger cardinals;
\item fails AP in all infinite cardinals;
\item there exist $2^{\lambda_i^+}$ many nonisomorphic maximal
    models in $\lambda_i^+$, for all $i\le\alpha$, but
    no maximal models in any other cardinality;
\item has arbitrarily large models; and
\item $LS(\bK^*)=\max\{LS(\bkzero^i)|i\le\alpha\}$.
\end{enumerate}
If all $(\bkzero^i,\prec_{\bkzero^i})$ are pure, then
$(\bK^*,\prec_{\bK^*})$ is pure. Further, if $\alpha < \aleph_1$ and
all $(\bkzero^i,\prec_{\bkzero^i})$ are definable by an
$\lomegaone$-sentence, then the same is true for
$(\bK^*,\prec_{\bK^*})$.
\begin{proof} Let $\bK^*$ be the AEC defined by the following construction.
\begin{enumerate}[a)]
\item The sets $(A_i|i\le\alpha)$ and $(C_i|i<\alpha)$ are non-empty and partition
    the universe.
\item For each $i,j$ with $i<j\le\alpha$, the triples
    $(A_i,A_j,C_i)$ and $(A_i,C_j,C_i)$ satisfy $\sigma_1$.
\item For each $i\le\alpha$, $C_i$ is a model in $\bkzero^i$, which
    implies that $|C_i|\le\lambda_i$.
\end{enumerate}

Define for $\M,\N\in \bK^*$, $\M\prec_{\bK^*}\N$, if $\M\subset \N$
with respect to $\tau_0$ and $C_i^{\M}\prec_{\bkzero^i} C_i^{\N}$,
for all $i\le\alpha$.

The proof is like the proof of Theorem \ref{2max} with some easy modifications.
Observe that if for some $i$, $|C_i|=\lambda_i$ and $|A_i|=|C_i|^+$, then by Lemma
\ref{Lem:BasicCombinatorics} all $A_j$, $C_j$, $j> i$, are ``locked'' to have
size at most $|A_i|$, and by Theorem \ref{Thm:ManyMaxModels} there
are $2^{\lambda_i^+}$ many nonisomorphic maximal models in
$\lambda_i^+$. If $|A_i|=|C_i|$, then the cardinalities of $A_j,C_j$,
$j>i$ can be greater than $\lambda_i^+$. We leave the rest of the details to the reader.
\end{proof}
\end{Thm}

We need some background before getting specific applications
of the previous theorem. The next fact follows from Theorem 4.20 of \cite{BKL}
for $\aleph_r = \kappa$, noting that joint embedding holds in
$\aleph_{r-1}$. Indeed, $2$-AP in $\aleph_{r-2}$ implies $2$-AP of
models with cardinality $\aleph_{r-1}$ over models of cardinality
$<\aleph_{r-1}$ (or the empty set). This yields a complete sentence
$\phi_r$ whose class of models is denoted $At^r$; a similar argument
for the incomplete sentence with models $\hat \bK^r$ is in Theorem
4.3 of that paper.

\begin{Fact}\label{nicecharsm} Every cardinal $\kappa< \aleph_\omega$ is
characterized by a (complete) sentence of $L_{\omega_1,\omega}$ that
satisfies JEP$(< \kappa)$.
\end{Fact}

 We describe the next example, based on \cite{Morley65a}, in detail
 since the particular formulation is important.

\begin{Ex} Fix some countable ordinal $\alpha$ and let $\{\beta_n|n\in\omega\}$ list the ordinals less than
$\alpha$. Consider the vocabulary $\tau$ that contains a binary
relation $\in$, a unary function $r$ (for `rank') and constants
$(c_{\beta_n})_{n\in\omega}$. Let $\phi_\alpha$ be the conjunction of
the following:

\begin{itemize}

 \item $\forall x,\; x\in c_{\beta_n}\leftrightarrow
     \bigvee_{\beta_i\in\beta_n} x=c_{\beta_i}$, for each $n$;
  \item $\forall x,\; \bigvee_{n\in\omega} r(x)=c_{\beta_n}$;
  \item $r(c_{\beta_n})=c_{\beta_n}$, for each $n$;
  \item $\forall x,y,\; x\in y \rightarrow r(x)\in r(y)$; and
  \item $\forall x,y,\; (\forall z)((z\in x\leftrightarrow z\in
      y)\rightarrow x=y)$ (Extensionality).
\end{itemize}
Observe that $\phi_\alpha$ is an $\lomegaone(\tau)$-sentence. Let
$\bK_\alpha$ be the collection of all models of $\phi_\alpha$. If
$M\in \bK_\alpha$, then $M$ can be embedded into $V_\alpha$. In
particular, $|M|\le|V_\alpha|=\beth_\alpha$.
\end{Ex}

\begin{Fact}\label{Cor:MorleyAEC}
 For each $\alpha<\omegaone$, $(\bK_\alpha, \subseteq)$ satisfies the following.
\begin{enumerate}[(a)]
 \item $\bK_\alpha$ has a unique maximal model in cardinality
     $\beth_\alpha$, and no larger models;
 \item JEP$(\le\beth_\alpha)$ holds; and
 \item $LS(\bK_\alpha)=\aleph_0$.
\end{enumerate}
\end{Fact}

Note that under GCH up to $\aleph_\omega$,
 Fact~\ref{Cor:MorleyAEC} is stronger than
Fact~\ref{nicecharsm} since JEP$(<\kappa)$ is replaced by
JEP$(\leq\kappa)$.

\begin{Cor}\label{manymaxap} Here are three applications of
Theorem~\ref{manymax}.
\begin{enumerate}
\item If $<\lambda_i|i\le\alpha\le\omega>$  is any increasing sequence of cardinals below
    $\aleph_\omega$, then there exists an $\lomegaone$ sentence $\psi$
    \begin{enumerate}
\item whose models satisfy JEP$(<\lambda_0)$;
\item that fails AP in all infinite cardinals;
\item has $2^{\lambda_i^+}$ many nonisomorphic maximal models
    in $\lambda_i^+$, for all $i\le\alpha$, but no maximal
    models in any other cardinality, while     JEP fails for
    all larger cardinals; and
\item  has arbitrarily large models.
\end{enumerate}
    \item If $<\beth_{\alpha_i}|i\le\gamma<\omegaone>$ is a strictly increasing sequence,
then there exists an $\lomegaone$ sentence $\psi'$
\begin{enumerate}
\item whose models satisfy JEP$(\le\beth_{\alpha_0})$;
\item  fails AP in all infinite cardinals;
\item has $2^{\beth_{\alpha_i}^+}$ many nonisomorphic maximal
    models in $\beth_{\alpha_i}^+$, for all $i\le\gamma$, but
    no maximal models in any other cardinality, while JEP
    fails for all larger cardinals; and
\item has arbitrarily large models.
\end{enumerate}
\item If $<\lambda_i|i\le\alpha\le\omega>$ is any countable increasing sequence of
    cardinals below $\beth_{\omega_1}$ that are characterized by
    complete $\lomegaone$ sentences,
    then there exists an $\lomegaone$-sentence $\psi''$
\begin{enumerate}
\item  whose models satisfy JEP$(\aleph_0)$;
\item fails AP in all infinite cardinals;
\item has $2^{\lambda_i^+}$ many nonisomorphic maximal models
    in $\lambda_i^+$, for all $i\le\alpha$, but no maximal models
    in any other cardinality; and
\item  has arbitrarily large models.
\end{enumerate}
\end{enumerate}
\end{Cor}

\begin{proof} For 1) use Theorem~\ref{manymax} and
Fact~\ref{nicecharsm}. For 2) use Theorem~\ref{manymax} and
Fact~\ref{Cor:MorleyAEC}.  3) is easy by Theorem~\ref{manymax} since
every complete sentence satisfies JEP in $\aleph_0$.
\end{proof}

\begin{question}  Is there an $\lomegaone$-sentence that has maximal models in uncountably many
cardinals but arbitrarily large models?
\end{question}

\bibliography{ssgroups}
\bibliographystyle{alpha}

\end{document}